\documentclass[12pt,a4paper, reqno]{amsart}
\usepackage[abbrev]{amsrefs}
\usepackage{amssymb,mathrsfs}
\usepackage{xcolor}
\usepackage{mathtools}

\usepackage{etoolbox}

\makeatletter
\patchcmd{\@settitle}{\uppercasenonmath\@title}{}{}{}
\patchcmd{\@setauthors}{\MakeUppercase}{}{}{}
\patchcmd{\section}{\scshape}{}{}{}
\makeatother

\theoremstyle{plain}
 \newtheorem{theorem}{Theorem}[section]
 \newtheorem{lemma}[theorem]{Lemma}
 
 \newtheorem{proposition}[theorem]{Proposition}

\theoremstyle{definition}
 \newtheorem{definition}[theorem]{Definition}

\theoremstyle{remark}

\DeclareMathOperator{\Hess}{Hess}
\DeclareMathOperator{\tr}{tr}
\DeclareMathOperator{\diag}{diag}
\numberwithin{equation}{section}

\renewcommand{\labelenumi}{{\rm(\theenumi)}}
\numberwithin{equation}{section}

\begin{document}
\title[Dirichlet heat flow on Riemannian manifolds]{\large Non-preservation of concavity properties\\ by the Dirichlet heat flow on Riemannian manifolds}
\author[]{Kazuhiro Ishige, Asuka Takatsu, and Haruto Tokunaga}
\date{}
\maketitle
\begin{abstract}
We prove that
no concavity properties are preserved  by the Dirichlet heat flow in a totally convex domain of a Riemannian manifold
unless the sectional curvature vanishes everywhere on the domain.
\end{abstract}
\vspace{10pt}

\vspace{15pt}
\noindent 
Addresses:

\smallskip
\noindent 
K.I.: Graduate School of Mathematical Sciences, The University of Tokyo,\\ 
3-8-1 Komaba, Meguro-ku, Tokyo 153-8914, Japan\\
\noindent 
E-mail: {\tt ishige@ms.u-tokyo.ac.jp}\\

\smallskip
\noindent 
A.T.: Department of Mathematical Sciences, Tokyo Metropolitan University,\\
1-1 Minami-osawa, Hachioji-shi, Tokyo 192-0397, Japan\\
\noindent 
E-mail: {\tt asuka@tmu.ac.jp}\\

\smallskip
\noindent 
H.T.: Graduate School of Mathematical Sciences, The University of Tokyo,\\ 
3-8-1 Komaba, Meguro-ku, Tokyo 153-8914, Japan\\
\noindent 
E-mail: {\tt okunagasp10@gmail.com}\\



\newpage
\section{Introduction}
Let $M$ be a complete, connected, and smooth Riemannian manifold without boundary of dimension at least $2$
and $\Omega$ a non-empty totally convex domain of $M$.
Consider the Cauchy--Dirichlet problem for the heat equation 
\begin{equation}
\tag{H}
\label{eq:H}
\begin{dcases}
\frac{\partial}{\partial t}u=\Delta_M u & \text{in}\quad\Omega\times(0,\infty),\\
u=0\quad & \text{in}\quad\partial\Omega\times(0,\infty)\text{ if }\partial\Omega\not=\emptyset,\\
u(\cdot,0)=\phi & \text{in}\quad\Omega,
\end{dcases}
\end{equation}
where $\phi$ is a nonnegative $L^\infty(\Omega)$-function.
Problem~\eqref{eq:H} possesses a unique (nonnegative and minimal) solution 
\[
u\in C^{2,1}(\Omega\times(0,\infty))\cap C(\overline{\Omega}\times(0,\infty))
\text{ satisfying }
\lim_{t\to +0} \|u(\cdot, t)-\phi(\cdot)\|_{L^2(\Omega\,\cap\, K)}=0
\]
for all compact sets $K$ of $M$ (see e.g.\,\cite{LSU}*{Chapters~III and IV}).
We call the solution~$u$ the \emph{Dirichlet heat flow} (which is abbreviated as DHF) in $\Omega$.
This paper is concerned with the preservation of concavity properties by DHF.

The preservation of concavity properties by DHF is a classical subject and it has fascinated many mathematicians 
since the pioneering work by Brascamp and Lieb~\cite{BL}. 
They found the preservation of log-concacvity by DHF in $\mathbb{R}^n$ via the Pr\'ekopa--Leindler inequality.
Subsequently, the preservation of log-concavity by DHF in convex domains of $\mathbb{R}^n$ was elaborated as follows
(see e.g.\,\cites{GK, INS, Keady, Kor}).
\begin{itemize}
\setlength{\leftskip}{-10pt}
\item
Let $u$ be a solution to problem~\eqref{eq:H} in a convex domain of $\mathbb{R}^n$.
If $\log\phi$ is concave in the domain, then $\log u(\cdot,t)$ is concave in the domain for every $t>0$.
\end{itemize}
Here and in what follows, we adhere to the convention that $\log 0\coloneqq -\infty$. 
The preservation of concavities properties on ${\mathbb R}^n$  has been studied   for various parabolic equations such as 
nonlinear heat equations~\cites{GK, INS, Ken02}; the porous medium equation~\cites{BV, DH1, DH2, DHL}; the evolution of $p$-Laplace operator~\cite{L}; 
fully nonlinear nonlocal parabolic equations~\cite{KLM}. 
See also \cites{Kawohl01, Kawohl, Kawohl02, Kor} for related topics.
In contrast, the non-preservation of concavity properties in convex domains of $\mathbb{R}^n$ was obtained 
in the following cases: 
the heat equation~\cites{CW1, IS01, IS02}; the heat equation with variable coefficients~\cites{AI, Kol}; 
the one-phase Stefan problem~\cite{CW2}; the porous medium equation~\cites{CW3, IS02}; 
the quasi-static droplet model~\cite{CW5}. 
Recently, in \cite{IST05}, 
the notion of $\overline{F}$-concavity was introduced as the largest available generalization of the notion of concavity on $\mathbb{R}^n$, 
and $\overline{F}$-concavity preserved by DHF in convex domains of $\mathbb{R}^n$ was characterized completely. 

Compared with the Euclidean space, much less is known about the preservation of concavity properties by DHF on Riemannian manifolds, even of space forms, 
and the following question naturally arises. 
\makeatletter\tagsleft@true\makeatother
\begin{equation}
\tag{Q}
\label{eq:Q}
\begin{split}
 & \mbox{Are there any concavity properties preserved by DHF in a totally convex domain}\\
 & \mbox{of a Riemannian manifold?} 
\end{split}
\end{equation}
\makeatletter\tagsleft@false\makeatother
Concerning this question, Shih~\cite{Shih}*{Theorem~1.5} found a totally convex domain of $\mathbb{H}^2$ 
whose first Dirichlet eigenfunction  of $-\Delta_{\mathbb{H}^2}$  is not quasi-concave.
Since quasi-concavity is the weakest conceivable concavity property,  this result together  with the Fourier expansion of DHF implies that no concavity properties are preserved by DHF in this domain.
Note that the first Dirichlet eigenfunctions in bounded convex domains of~$\mathbb{R}^n$ are log-concave 
due to the preservation of log-concavity by DHF, 
and those in totally convex domains of~${\mathbb S}^n$ are also log-concave (see e.g.~\cites{LW}).
We also refer to \cite{IST02}*{Theorem~1.2} for the preservation of log-concavity by {rotationally symmetric} DHF on a Riemannian manifold.

In this paper we investigate the preservation of concavity properties by DHF in~$\Omega$ to obtain our answer to question~\eqref{eq:Q}.  
Surprisingly, the answer is that no concavity properties are preserved by DHF in~$\Omega$ unless the sectional curvature vanishes everywhere on~$\Omega$ 
(see Theorem~\ref{theorem:1.4}). 

We introduce some notation to state our answer precisely. 
We first recall the definition of $F$-concavity.
\begin{definition}
\label{theorem:1.1}
Let $a\in(0,\infty]$.
\begin{itemize}
\setlength{\leftskip}{-10pt}
\item[{\rm (1)}]
A function~$F\colon [0,a)\to [-\infty,\infty)$ is
\emph{admissible} on $[0,a)$ if $F\in C((0,a))$, 
$F$ is strictly increasing on $[0,a)$, and $F(0)=-\infty$.
\item[{\rm (2)}]
Let $F$ be admissible on $[0,a)$.
Define
$$
\mathcal {A}_\Omega(a)\coloneqq \{f\colon \Omega \to \mathbb{R} \mid f(\Omega) \subset [0,a)\}.
$$
For $f\in\mathcal {A}_\Omega(a)$,
we say that $f$ is \emph{$F$-concave} in $\Omega$
if $F\circ f$ is concave in $\Omega$, more precisely, 
for any minimal geodesic $c\colon [0,1]\to \Omega$, 
\begin{align}\label{eq:1.1}
F(f(c_\tau))
\geq
 (1-\tau) F(f(c_0))+\tau F(f(c_1))
 \quad \text{for all $\tau\in(0,1)$}.
\end{align}
We denote by ${\mathcal C}_\Omega[F]$ the set of $F$-concave functions in $\Omega$.
\item[{\rm (3)}]
For $i=1,2$, let $F_i$ be admissible on $[0,a_i)$ with $a_i\in(0,\infty]$ and $a\in (0,\min\{a_1,a_2\}]$.
We say that \emph{$F_1$-concavity is  weaker} (resp.\,\emph{strictly weaker}) \emph{than $F_2$-concavity}, 
or equivalently that \emph{$F_2$-concavity is stronger} (resp.\,\emph{strictly stronger})
\emph{ 
than $F_1$-concavity}, in $\mathcal{A}_\Omega(a)$ if
\[
\mathcal{C}_\Omega[F_2]\cap\mathcal{A}_\Omega(a)\subset\mathcal{C}_\Omega[F_1] 
\quad
\mbox{{\rm (}resp.\,\,$\mathcal{C}_\Omega[F_2]\cap{\mathcal A}_\Omega(a)\subsetneq\mathcal{C}_\Omega[F_1]\cap\mathcal{A}_\Omega(a)${\rm)}}.
\]
\end{itemize}
\end{definition}
As a typical example of $F$-concavity, 
we introduce $\alpha$-concavity, where $\alpha\in\mathbb{R}$. 
Define an admissible function $\Phi_\alpha$ on $[0,\infty)$ by 
\begin{equation*}
\Phi_\alpha(r)\coloneqq
\begin{dcases}
\frac{1}{\alpha}(r^\alpha-1) & \text{for $r>0$ if $\alpha\not=0$},\vspace{5pt}\\
\log r & \text{for $r>0$ if $\alpha=0$},\vspace{5pt}\\
-\infty & \text{for $r=0$}.
\end{dcases}
\end{equation*}
We refer to $\Phi_\alpha$-concavity as \emph{$\alpha$-concavity}.
Note that $0$-concavity corresponds to log-concavity. 
We also define $\pm\infty$-concavity as follows. 
\begin{definition}
\label{theorem:1.2}
A nonnegative function $f$ in $\Omega$ is 
\emph{$\infty$-concave} (resp.\,\emph{$-\infty$-concave}) if 
\[
 f(c_\tau )\ge\max\{f(c_0),f(c_1)\}
 \quad
 \mbox{(resp.~$f(c_\tau)\ge\min\{f(c_1),f(c_1)\}$)}
\]
for all $\tau \in (0,1)$ and  minimal geodesic $c\colon [0,1]\to \Omega$ with $f(c_0)$, $f(c_1)>0$.
\end{definition}
We often refer to $-\infty$-concavity as quasi-concavity. 
Power concavity is a generic term for $\alpha$-concavity with $\alpha\in[-\infty,\infty]$ 
and it possesses the following property via Jensen's inequality.
\begin{itemize} 
\setlength{\leftskip}{-10pt}
\item 
For any $\alpha, \beta \in [-\infty,\infty]$ with $\beta \leq \alpha$,
an $\alpha$-concave function in $\Omega$ is $\beta$-concave in $\Omega$.
\end{itemize}
This property establishes a hierarchy among power concavities 
and quasi-concavity is the weakest one while $\infty$-concavity is the strongest one.
Furthermore, an $F$-concave function is quasi-concave and an $\infty$-concave function is $F$-concave for every admissible function~$F$.
Thus quasi-concavity (resp.\,$\infty$-concavity) remains the weakest (resp.\,strongest) conceivable concavity.
Since both concavities do not posses any corresponding admissible functions,
we use the expression \emph{$\overline{F}$-concavity} when we consider all the $F$-concave functions jointly with quasi-concave and $\infty$-concave functions.
Then $\overline{F}$-concavity is the largest available generalization of concavity on~$\mathbb{R}^n$.

We define the notion of the preservation of $\overline{F}$-concavity by DHF.  
\begin{definition}
\label{theorem:1.3}
Denote by $e^{t\Delta_\Omega}\phi$ the unique nonnegative and minimal solution to problem~\eqref{eq:H}.
We say that \emph{$\overline{F}$-concavity is preserved by DHF in~$\Omega$} 
if, for any bounded and $\overline{F}$-concave function $\phi$ in~$\Omega$, 
$e^{t\Delta_\Omega}\phi$ is $\overline{F}$-concave in~$\Omega$ for every $t>0$.
\end{definition}

Note that the maximum principle for problem~\eqref{eq:H} implies that 
\[
\text{$e^{t\Delta_\Omega}\phi\in\mathcal{A}_\Omega(a)\cap L^\infty(\Omega)$ holds for every $t>0$ if $\phi\in\mathcal{A}_\Omega(a)\cap L^\infty(\Omega)$}, 
\]
where $a\in (0,\infty]$. 

Now we are ready to state our main result, which gives our negative answer to question~\eqref{eq:Q}.
\begin{theorem}
\label{theorem:1.4}
If some $\overline{F}$-concavity is preserved by DHF in $\Omega$,
then the sectional curvature is identically zero on $\Omega$. 
\end{theorem}
Thus no concavity properties are preserved  by the Dirichlet heat flow in a curved totally convex domain. 
While at the same time there are several notions of convexity for sets and concavity of functions due to the non-uniqueness of minimal geodesics
so that there is room to answer question~\eqref{eq:Q} with other notions. 
Recall that a subset $A$ of $M$ is \emph{totally convex} if any minimal geodesic in $M$ joining two points in $A$ also lies in $A$.
On one hand, as for a stronger convexity, $A$ is \emph{strongly convex} 
if for any two points in $A$ there exists a unique minimal geodesic joining them in $M$  and the geodesic is contained in $A$ (see e.g.\,\cite{Sakai}*{Definition~IV.5.1}).
On the other hand, as for a weaker one, 
$A$ is \emph{convex} if every two points in $A$ there is a minimal geodesic joining  them which entirely belongs to $A$
(see e.g.\,\cite{BBI}*{Definition~3.6.5}).
Theorem~\ref{theorem:1.4} obviously holds for strongly convex domains 
and also for  convex domains since any open convex set is  totally convex.
We can also discuss weak $\overline{F}$-concavity.
Recall that a function $f$ in a convex  is  \emph{weakly $F$-concave} in $A$
if for any two points in $A$ there exists at least one minimal geodesic $c\colon [0,1]\to A$ joining them such that \eqref{eq:1.1} holds
(see e.g.\,\cite{Vi}*{Definition~16.1}).
The weak $\pm\infty$-concavity  is also defined in a similar way.
Then $\overline{F}$-concavity and  weak $\overline{F}$-concavity are equivalent to each other in a strongly convex domain 
but not  necessarily in a totally convex domain.
In addition, the preservation of $\overline{F}$-concavity by DHF in a totally convex domain 
may not always imply that of  weak $\overline{F}$-concavity and vice versa. 
However, with our method, Theorem~\ref{theorem:1.4} holds even if we replace $\overline{F}$-concavity with  weak $\overline{F}$-concavity 
(see Theorem~\ref{theorem:4.2}). 
Thus our answer to question~\eqref{eq:Q} is still negative even if we treat strongly convex domains, convex domains, and weak $\overline{F}$-concavity.
There is still another notion of concavity for functions. 
As mentioned before, the  Pr\'ekopa--Leindler inequality plays an important role in the proof of the preservation of log-concavity in~$\mathbb{R}^n$
and there is a Riemannian version of the Pr\'ekopa--Leindler inequality inequality (see e.g.\,\cite{Vi}*{Theorem~19.16}).
The Pr\'ekopa--Leindler inequality on a Riemannian manifold  incorporates a distortion coefficient, not linear coefficients, due to the curvature of a Riemannian manifold,
which suggests the possibility of affirmative  to question~\eqref{eq:Q} with the definition of concavity involving distortion coefficients.

We explain the strategy of the proof of Theorem~\ref{theorem:1.4}.
Assume that $\overline{F}$-concavity is preserved by DHF in $\Omega$.
We apply similar transformation of the heat equation on~$\mathbb{R}^n$ to prove that 
$\overline{F}$-concavity is also preserved by DHF in $\mathbb{R}^n$ (see Proposition~\ref{theorem:3.2}). 
In this step we require delicate approximations of DHF.
Next, we apply the arguments in \cite{IST05} to obtain a characterization of $\overline{F}$-concavity, and we deduce that 
log-concavity is  also preserved by DHF in $\Omega$ (see Proposition~\ref{theorem:3.5}).
Finally, we employ the arguments in \cite{CW3} (see Lemma~\ref{theorem:4.1}) to find a log-concave initial data 
for which the corresponding solution to problem~\eqref{eq:H} is not log-concave for some $t>0$ 
if the sectional curvature does not vanish on $\Omega$.
Then we see that no $\overline{F}$-concavity is preserved by DHF in $\Omega$
unless the sectional curvature vanishes everywhere on $\Omega$, and complete the proof of Theorem~\ref{theorem:1.4}.

The rest of this paper is organized as follows.
Section~\ref{section:2} concerns with the notation and  some properties of $\overline{F}$-concavity.
Section~\ref{section:3} is devoted to a characterization of $\overline{F}$-concavity preserved by DHF in totally convex domains of a Riemannian manifold.
In Section~\ref{section:4} we prove Theorem~\ref{theorem:1.4}. 
Furthermore, we obtain non-preservation of weak $\overline{F}$-concavity.
\section{Preliminary}\label{section:2}
In this section we fix our notation and recall some properties of $\overline{F}$-concavity.
Throughout this paper,
let  $n\geq 2$ and  $M=(M,g)$ be a complete, connected, smooth $n$-dimensional Riemannian manifold without boundary.
We denote by $x^1,\ldots, x^n$ the natural coordinate functions of $\mathbb{R}^n$.
For $o\in M $ and $r>0$,
let $B^{M}_o(r)$ denotes the open ball in $M$ centered at~$o$ of radius $r$.
%
%
Let ${\bf 1}_E$ denote the characteristic function of a set $E$.

%

Let us first recall some expressions in terms of a coordinate system~$(U,\xi)$ in $M$.
As usual, setting  
\begin{align*}
R_{ijk\ell}
&\coloneqq 
g\left(\frac{\partial}{\partial \xi^i}, 
D_{\frac{\partial}{\partial \xi^\ell}}\left({D_\frac{\partial}{\partial \xi^{k}}}\frac{\partial}{\partial \xi^j}\right)
-D_{\frac{\partial}{\partial \xi^k}}\left({D_\frac{\partial}{\partial \xi^{\ell}}}\frac{\partial}{\partial \xi^j}\right)
\right)
\quad\text{for }i,j,k,\ell=1,\ldots,n,
\end{align*}
where $D$ is the Levi-Civita connection of $M$,
we have $R_{ijk\ell}=-R_{jik\ell}=-R_{ij\ell k}=R_{k\ell ij}$.
The coordinate expression of the gradient vector field, Hessian, and the Laplacian of a smooth function~$f$ on $U$ are given by
\begin{align*}
 & \nabla_Mf=\sum_{i,j=1}^n g^{ij} \frac{\partial f}{\partial \xi^j} \frac{\partial }{\partial \xi^i},\\
 & \Hess_M f\left( \frac{\partial }{\partial \xi^i},\frac{\partial }{\partial \xi^j} \right) 
=\frac{\partial^2 f}{\partial \xi^i\partial \xi^j}-\sum_{k=1}^n\Gamma_{ij}^k \frac{\partial f}{\partial \xi^k} \quad \text{for }i,j=1,\ldots,n,\\
 & \Delta_Mf 
= \sum_{i,j=1}^n g^{ij}\left( \frac{\partial^2 f}{\partial \xi^i\partial \xi^j}-\sum_{k=1}^n\Gamma_{ij}^k \frac{\partial f}{\partial \xi^k}\right),
\end{align*}
respectively.
If $(U,\xi)$ is a normal coordinate system at $o$ determined by the orthonormal basis $e_1,\ldots, e_n$ for $T_oM$, then, for $i,j,k=1,\ldots,n$, we have 
\begin{align}
\begin{split}
\label{eq:2.1}
e_i&=\frac{\partial}{\partial \xi^i}\bigg|_o,\\
g_{ij}&=g\left( \frac{\partial }{\partial \xi^i},\frac{\partial }{\partial \xi^j} \right) =\delta_{ij}+\frac{1}{3}\sum_{k,\ell=1}^n R_{ikj\ell}(o)\xi^k\xi^{\ell}+O(|\xi|^3),\\
g^{ij}&=\delta_{ij}-\frac{1}{3}\sum_{k,\ell=1}^n R_{ikj\ell}(o)\xi^k\xi^{\ell}+O(|\xi|^3),\\
\Gamma_{ij}^k&=\frac13\sum_{\ell=1}^n \left(R_{i \ell k j }(o) +R_{j \ell k i }(o) \right)  \xi^{\ell} +O(|\xi|^2),
\end{split}
\end{align}
as $\xi \to0$,
which in turn yields  $g_{ij}(o)=g^{ij}(o)=\delta_{ij}$ and $\Gamma_{ij}^k(o)=0$
(see e.g.\,\cite{Sakai}*{Proposition~II~3.1}).
For distinct $i,j=1,\ldots,n$, $-R_{ijij}(o)$ is the sectional curvature of the tangent plane spanned by  $e_i$ and $e_j$.

Next, we introduce hot-concavity (see~\cite{IST05}).
\begin{definition}
\label{theorem:2.1}
Let $H\colon (0,1)\to \mathbb{R}$ be the inverse function of
\[
\left(e^{{\Delta}_\mathbb{{R}}}{\bf 1}_{[0,\infty)}\right)(s)=(4\pi)^{-\frac{1}{2}}\int_0^\infty e^{-\frac{(s-s')^2}{4}}\,ds'
\quad\mbox{for $s\in\mathbb{R}$}.
\]
Given $a\in(0,\infty]$, we define an admissible function $H_a$ on $[0,a)$ by 
\begin{equation*}
H_a(r)\coloneqq 
\begin{cases}
H(a^{-1}r) & \mbox{for $r\in(0,a)$ if $a>0$},\vspace{3pt}\\
\log r & \mbox{for $r\in(0,a)$ if $a=\infty$},\vspace{3pt}\\
-\infty & \mbox{for $r=0$}.
\end{cases}
\end{equation*}
\emph{Hot-concavity} is a generic term for $H_a$-concavity with $a\in (0,\infty]$.
\end{definition}
We  recall a result on a characterization of $\overline{F}$-concavity preserved by DHF in convex domains of $\mathbb{R}^n$. 
\begin{proposition}[\cite{IST05}*{Theorem~1.5, Theorem~5.2}]
\label{theorem:2.2}
Let $\Omega$ be a convex domain of $\mathbb{R}^n$.
Let $F$ be admissible on $[0,a)$ with $a\in(0,\infty]$ 
and assume that $F$-concavity is preserved by DHF in~$\Omega$.
\begin{enumerate}
\renewcommand{\labelenumii}{(\theenumii)}
\setlength{\leftskip}{-10pt}
\item \label{enumi:1}
$F$-concavity is weaker than $H_a$-concavity 
and 
stronger than log-concavity  in $\mathcal{A}_\Omega(a)$.
\item\label{enumi:2} 
If a function $f$ is $F$-concave in $\Omega$, then so is $\varepsilon f$ for $\varepsilon\in(0,1)$.
\item \label{enumi:3}
If $F$-concavity is strictly weaker than log-concavity in ${\mathcal A}_\Omega(a)$, 
  then there exists a bounded continuous function $\phi$ on $\overline{\Omega}$ 
  with the following  properties:
  \begin{itemize}
  \setlength{\leftskip}{-10pt}
  \item[$\bullet$]
   $\phi$ is $F$-concave in $\Omega$ and $\phi=0$ on $\partial\Omega$ if $\partial\Omega\not=\emptyset$;
  \item[$\bullet$]
   $e^{t\Delta_\Omega}\phi$ is not quasi-concave in $\Omega$ for some $t>0$.
\end{itemize}
\end{enumerate}
\end{proposition}
Proposition~\ref{theorem:2.2}~(1) together with the following lemma implies that 
if an admissible function $F$ satisfies $\lim_{r\to 0+}F(r)>-\infty$, 
then  $F$-concavity is not preserved by Dirichlet heat flow in $\mathbb{R}^n$. 
Note that $\lim_{r\to +0}H_a(r)=-\infty$.
\begin{lemma}[\cite{IST05}*{Lemma~2.7}]
\label{theorem:2.4}
Let $F_1$ and $F_2$ be admissible on $I=[0,a)$ with $a\in(0,\infty]$. 
If $F_1$-concavity is weaker than $F_2$-concavity in ${\mathcal A}_{{\mathbb R}^n}(a)$ 
and $\lim_{r\to +0}F_2(r)=-\infty$, 
then $\lim_{r\to +0}F_1(r)=-\infty$.
\end{lemma}

The following lemmas  are shown similarly to the argument for \cite{IST05}*{Lemma~2.6} and \cite{IST05}*{Lemma~2.10}, respectively. 
We omit the proofs.
\begin{lemma}[]
\label{theorem:2.3}
Let $F_1$ and $F_2$ be admissible on $[0,a)$ with $a\in (0,\infty]$.
For a totally convex domain $\Omega$ of $M$, 
\begin{equation*}
 {\mathcal C}_\Omega[F_2]\subset{\mathcal C}_\Omega[F_1]
\quad\mbox{if and only if}\quad
{\mathcal C}_\mathbb{{R}}[F_2]\subset{\mathcal C}_\mathbb{{R}}[F_1].
\end{equation*}
Similarly, any weakly $F_2$-concave function in $\Omega$ is weakly $F_1$-concave in $\Omega$ 
if and only if ${\mathcal C}_\mathbb{{R}}[F_2]\subset{\mathcal C}_\mathbb{{R}}[F_1]$.
\end{lemma}
\begin{lemma}
\label{theorem:2.5}
Let $\Omega$ be  a totally convex domain $\Omega$ of $M$.
For any bounded log-concave \textup{(}resp.\,weakly log-concave\textup{)} function $f$ in $\Omega$,
there exists a sequence $\{f_a\}_{a>0}$ such that 
$f_a$ is $H_a$-concave \textup{(}resp.\,weakly $H_a$-concave\textup{)} in $\Omega$ and 
$\lim_{a\to\infty}f_a=f$ uniformly on $\Omega$.
\end{lemma}
%
\section{Characterization of $\overline{F}$-concavity preserved by DHF}\label{section:3}
In this section we give a characterization of $\overline{F}$-concavity preserved by DHF in convex domains.
Throughout the rest of this paper, 
let $\Omega$ be a totally convex domain of $M$ 
and $F$ an admissible on $[0,a)$ with $a\in (0,\infty]$ unless otherwise indicated.
Recall that, given nonnegative $\phi\in L^\infty(\mathbb{R}^n)$,
\[
(e^{t\Delta_{\mathbb{R}^n}}\phi)(x)=(4\pi t)^{\frac{n}{2}}\int_{\mathbb{R}^n}e^{-\frac{|x-y|^2}{4t}}\phi(y)\,dy 
\quad\text{for } (x,t)\in\mathbb{R}^n\times(0,\infty).
\]
We first prove that
the preservation of $\overline{F}$-concavity by DHF in $\Omega$ implies that in $\mathbb{R}^n$. 
\begin{lemma}
\label{theorem:3.2}
If $F$-concavity is not preserved by DHF in $\mathbb{R}^n$ and $\lim_{r\to +0}F(r)=-\infty$, 
then neither $F$-concavity nor weak $F$-concavity is  preserved by DHF in $\Omega$.
\end{lemma}
\begin{proof}
We first assume that $F$-concavity is not preserved by DHF in $\mathbb{R}^n$, 
that is, there exists $\psi\in {\mathcal C}_{\mathbb{R}^n}[F]\cap L^\infty(\mathbb{R}^n)$ such that 
$$
F\left((e^{t_*\Delta_{\mathbb{R}^n}}\psi)((1-\tau)y+\tau z)\right)
<(1-\tau)F\left((e^{t_*\Delta_{\mathbb{R}^n}}\psi)(y)\right)
+\tau F\left((e^{t_*\Delta_{\mathbb{R}^n}}\psi)(z)\right)
$$
for some $y$, $z\in\mathbb{R}^n, t_\ast >0$, and $\tau\in(0,1)$. 
Without loss of generality, 
we can assume that $y$, $z$ and the origin are on the same straight line.

 Let $r_*>0$ be such that $\psi_*=\psi\,{\bf 1}_{B_0^{\mathbb{R}^n}(r_*)}$ is $F$-concave in $\mathbb{R}^n$
and 
\begin{equation}\label{eq:3.1}
 F\left((e^{t_*\Delta_{\mathbb{R}^n}}\psi_*)((1-\tau)y+\tau z)\right)
  <(1-\tau)F\left((e^{t_*\Delta_{\mathbb{R}^n}}\psi_*)(y)\right)
+\tau F\left((e^{t_*\Delta_{\mathbb{R}^n}}\psi_*)(z)\right).
\end{equation}
Since $e^{F(\psi_*)}$ is log-concave in $\mathbb{R}^n$,
the preservation of log-concavity together with the strong maximum principle and the smoothing effect by the heat flow in~$\mathbb{R}^n$ implies that 
\begin{equation}
\label{a}
\mbox{$\log\left(e^{t\Delta_{\mathbb{R}^n}}e^{F(\psi_*)}\right)$ is smooth and concave in $\mathbb{R}^n$ for every $t>0$}.
\end{equation}
Furthermore,  by the fact that
\[
\lim_{t\to +0}\left\|e^{t\Delta_{\mathbb{R}^n}}e^{F(\psi_*)}-e^{F(\psi_*)}\right\|_{L^2(K)}=0
 \quad\text{ for all compact sets $K$ of $\mathbb{R}^n$},
 \]
we find a sequence $\{\varepsilon_m\}_{m\in\mathbb{N}}$ such that $\lim_{m\to\infty}\varepsilon_m=0$ and
\begin{equation}
\label{eq:3.2}
\lim_{m\to\infty} \left(e^{\varepsilon_m\Delta_{\mathbb{R}^n}}e^{F(\psi_*)}\right)(x)=e^{F(\psi_*(x))}
\quad\text{for almost all $x\in\mathbb{R}^n$}.
\end{equation}

Now we assume that $\lim_{r\to +0}F(r)=-\infty$, and set 
\[
\psi_m(x)\coloneqq 
F^{-1}\left(\log\left(\left(e^{\varepsilon_m\Delta_{\mathbb{R}^n}}e^{F(\psi_*)}(x)\right)\right)-\varepsilon_m |x|^2)\right)\quad\text{for }x\in{\mathbb R}^n.
\]
Then 
$\psi_m$ is $F$-concave  in $\mathbb{R}^n$. 
In particular, for any $v\in \mathbb{R}^n$, we observe from \eqref{a} that
\begin{equation}\label{eq:3.3}
\Hess_{\mathbb{R}^n}( F\circ \psi_m) (v,v)\le -2\varepsilon_m|v|^2 
\quad \text{in } {\mathbb R}^n.
\end{equation}
Since Lebesgue's dominated convergence theorem with \eqref{eq:3.2}
implies
\begin{align*}
\lim_{m\to\infty}\left(e^{t\Delta_{\mathbb{R}^n}}\psi_m\right)(x)
 & =(4\pi t)^{-\frac{n}{2}}\lim_{m\to\infty}\int_{\mathbb{R}^n}
\exp\left(-\frac{|x-x'|^2}{4t}\right)\psi_m(x')\,dx'
 =\left(e^{t\Delta_{\mathbb{R}^n}}\psi_*\right)(x)
\end{align*}
for $(x,t)\in\mathbb{R}^n\times(0,\infty)$, 
 we observe from \eqref{eq:3.1} that 
\begin{equation}
\label{eq:3.4}
 F\left((e^{t_*\Delta_{\mathbb{R}^n}}\psi_m)((1-\tau)y+\tau z)\right)
 <(1-\tau)F\left((e^{t_*\Delta_{\mathbb{R}^n}}\psi_m)(y)\right)
+\tau F\left((e^{t_*\Delta_{\mathbb{R}^n}}\psi_m)(z)\right)
\end{equation}
for every large enough $m\in \mathbb{N}$.
We fix such $m$.
For $x\in\mathbb{R}^n$, we calculate 
\begin{equation*}
\begin{split}
 \left|\left(\nabla_{\mathbb{R}^n} e^{\varepsilon_m\Delta_{\mathbb{R}^n}}e^{F(\psi_*)}\right)(x)\right|
  & \le\frac{(4\pi \varepsilon_m)^{-\frac{n}{2}}}{2\varepsilon_m}
\int_{B_0^{\mathbb{R}^n}(r_*)}|x-x'|\exp\left(-\frac{|x-x'|^2}{4\varepsilon_m}\right)e^{F(\psi(x'))}\,dx'\\
 & \le\frac{(4\pi \varepsilon_m )^{-\frac{n}{2}}}{2\varepsilon_m}(|x|+r_*)
\int_{B_0^{\mathbb{R}^n}(r_*)}\exp\left(-\frac{|x-x'|^2}{4\varepsilon_m}\right)e^{F(\psi(x'))}\,dx'\\
 & =\frac{|x|+r_*}{2\varepsilon_m}\left(e^{\varepsilon_m\Delta_{\mathbb{R}^n}}e^{F(\psi_*)}\right)(x),
\end{split}
\end{equation*}
which yields
\begin{align}
\label{eq:3.5}
|\nabla_{\mathbb{R}^n} (F\circ \psi_m)(x)|
=
\left|\left(\nabla_{\mathbb{R}^n} \log\left(e^{\varepsilon_m\Delta_{\mathbb{R}^n}}e^{F(\psi_*)}\right)\right)(x)
-2\varepsilon_mx \right|
\leq 
\frac{|x|+r_*}{2\varepsilon_m}+2\varepsilon_m|x|.
\end{align}

Fix $o\in \Omega$ and choose $r>0$ such that 
$B_o^M(r)$ is strongly convex with $B_o^M(r)\subset \Omega$.
Then~$B_o^M(r) $ becomes a normal neighborhood of $o$. 
Since all the Christoffel symbols for the normal coordinate system $(B_o^M(r),\xi)$ 
vanish at $o$, the quantity 
\[
\Lambda(r)\coloneqq \sup_{p\in B_o^M(r)}\max_{1\leq k\leq n}\left\{
\text{the operator norm of the matrix $[\Gamma^k_{ij}(p)]_{1\leq i,j \leq n}$}\right\}
\]
goes to $0$ as $r\to+0$.
For $\lambda>1$, 
define $\phi_\lambda: \Omega\to \mathbb{R}$ by 
$$
\phi_\lambda(p)\coloneqq 
\left\{
\begin{array}{ll}
\psi_m(\lambda\xi(p)) & \mbox{for $p\in B_o^M(r)$},\vspace{3pt}\\
0 & \mbox{otherwise}.
\end{array}
\right.
$$
By the strong convexity of $B_o^M(r)$ 
$\phi_{\lambda}$ is $F$-concave in $\Omega$ if and only if $\phi_\lambda$ is $F$-concave in~$B_o^M(r)$,
which is 
equivalent to the nonpositive definiteness of $\Hess_M (F\circ \phi_\lambda)$ in $B_o^M(r)$.
For $p\in B_o^M(r)$, 
it follows from \eqref{eq:3.3} and \eqref{eq:3.5}  that
\begin{align*}
&\mathrm{Hess}_M (F\circ \phi_\lambda) \left(\sum_{i=1}^n v^i \frac{\partial}{\partial \xi^i}\bigg|_p,\sum_{j=1}^n v^j \frac{\partial}{\partial \xi^j}\bigg|_p\right)\\ 
& =\lambda^2 \mathrm{Hess}_M (F\circ \psi_m)|_{\lambda \xi(p)}(v,v)
-\lambda \sum_{i,j,k=1}^n \Gamma_{ij}^k(p)v^iv^j \frac{\partial (F\circ \psi_m)}{\partial x^k}(\lambda \xi(p))\\
 & \leq
-2\lambda^2 \varepsilon_m|v|^2 
+\lambda \Lambda(r) |v|^2 \left(\frac{\lambda r+r_*}{2\varepsilon_m}+2\varepsilon_m \lambda r\right)\\
 & <-2\lambda^2 \varepsilon_m|v|^2 
\left[1-\Lambda(r)  \left(\frac{r+r_*}{4\varepsilon_m^2}+r\right)\right]
\end{align*}
for $v=(v^1,\ldots,v^n)\in \mathbb{R}^n$.
This is negative if $r>0$ is small enough since $\Lambda(r)\to 0$ as $r\to +0$.
We fix such $r>0$ 
determined independent of $\lambda >1$.

For $i, j, k=1,\ldots,n$, define the functions 
$\tilde{a}_\lambda^{ij}$ and $\tilde{b}_{\lambda}^k$ on $B_0^{\mathbb{R}^n}(\lambda r)$ by 
\begin{align*}
\tilde{a}^{ij}_\lambda(x)\coloneqq g^{ij}(\xi^{-1}(\lambda^{-1}x) ),
\quad
\tilde{b}^{k}_{\lambda}(x)\coloneqq \lambda^{-1} \sum_{i,j=1}^n g^{ij}(\xi^{-1}(\lambda^{-1}x) )\Gamma_{ij}^k(\xi^{-1}(\lambda^{-1}x)),
\end{align*}
for $x\in B_0^{\mathbb{R}^n}(\lambda r)$.
Setting 
\begin{equation*}
\begin{array}{ll}
u_\lambda(p,t)\coloneqq \left(e^{t \Delta_\Omega} \phi_\lambda\right)(p) &\text{for } (p,t)\in\Omega \times (0,\infty),\vspace{5pt}\\
\tilde{u}_\lambda(x,t)\coloneqq u_\lambda(\xi^{-1}(\lambda^{-1}x),\lambda^{-2}t)
 &\text{for } (x,t)\in B_0^{\mathbb{R}^n}(\lambda r) \times (0,\infty),
\end{array}
\end{equation*}
we see that
\[
\begin{dcases}
\frac{\partial}{\partial t} \tilde{u}_\lambda=
\sum_{i,j=1}^n \tilde{a}^{ij}_\lambda\frac{\partial^2}{\partial x^i\partial x^j} \tilde{u}_\lambda
-\sum_{k=1}^n \tilde{b}^{k}_\lambda \frac{\partial }{\partial x^k}\tilde{u}_\lambda
& \text{in\ }B_0^{\mathbb{R}^n}(\lambda r)\times (0,\infty),\\
\tilde{u}_\lambda(\cdot,0)=\psi_m & \text{in\ } B_0^{\mathbb{R}^n}(\lambda r).
\end{dcases}
\]
Thanks to parabolic regularity theorem and the Ascoli--Arzel\'a theorem, 
applying the diagonal argument, we find a sequence $\{\lambda_\ell\}_{\ell\in\mathbb{N}}\subset(0,\infty)$ 
with $\lim_{\ell\to\infty}\lambda_\ell=\infty$ such that
\[
\lim_{\ell\to\infty} \tilde{u}_{\lambda_\ell}(x,t)=(e^{t\Delta_{\mathbb{R}^n}}\psi_m)(x)
\]
uniformly for all compact sets in $\mathbb{R}^n\times(0,\infty)$.
This together with \eqref{eq:3.4} implies that 
\begin{align}
\label{eq:3.6}
 F(\tilde{u}_{\lambda_\ell}((1-\tau)y+\tau z,t_*))
 <(1-\tau)F(\tilde{u}_{\lambda_\ell}(y,t_*))+\tau F(\tilde{u}_{\lambda_\ell}(z,t_*))
\end{align}
for every  large enough $\ell\in\mathbb{N}$ 
with $\lambda_\ell^{-1}y,\lambda_\ell^{-1}z\in B_0^{\mathbb{R}^n}(r)$.
Since  $y$, $z$, and the origin are on the same straight line,
\[
\tau\mapsto  \xi^{-1}\left(  (1-\tau)\lambda_\ell^{-1} y + \tau \lambda_\ell^{-1} z \right)\in B_o^M(r) \quad \text{for }\tau \in[0,1]
\]
forms a unique minimal geodesic from $\xi^{-1}(\lambda_\ell^{-1}y)$ to $\xi^{-1}(\lambda_\ell^{-1}z)$. 
Thus, by \eqref{eq:3.6} $u_{\lambda_\ell}(\cdot,\lambda_\ell^{-2}t_*)$ is 
not weakly $F$-concave in $\Omega$ hence  
neither $F$-concavity  nor weakly $F$-concavity is preserved by DHF in $\Omega$. 
Thus Lemma~\ref{theorem:3.2} follows.
\end{proof}
We modify the argument in the proof of Lemma~\ref{theorem:3.2} to prove the following two lemmas.
\begin{lemma}
\label{theorem:3.3}
Assume that $F$-concavity is strictly weaker than log-concavity in ${\mathcal A}_{{\mathbb R}^n}(a)$. 
Then there exists $\phi\in {\mathcal C}_\Omega[F]\cap L^\infty(\Omega)$ such that 
$e^{t_*\Delta_\Omega}\phi$ is not weakly quasi-concave in $\Omega$ for some $t_*>0$. 
In particular, 
neither quasi-concavity nor weak quasi-concavity is preserved by DHF in $\Omega$.
\end{lemma}
\begin{proof}
Assume that $F$-concavity is strictly weaker than log-concavity in ${\mathcal A}_{{\mathbb R}^n}(a)$. 
It follows from Lemma~\ref{theorem:2.4} that
\[
\lim_{r\to +0}F(r)=-\infty.
\]
By Proposition~\ref{theorem:2.2}~(3) we find $\psi\in {\mathcal C}_{{\mathbb R}^n}[F]\cap L^\infty({\mathbb R}^n)$ such that
\begin{equation}
\label{eq:3.7}
\left(e^{t_*\Delta_{{\mathbb R}^n}}\psi\right)((1-\tau)y+\tau z)
<\min\left\{\left(e^{t_*\Delta_{{\mathbb R}^n}}\psi\right)(y),\left(e^{t_*\Delta_{{\mathbb R}^n}}\psi\right)(z)\right\}
\end{equation}
for some $y$, $z\in{\mathbb R}^n$, $\tau\in(0,1)$, and $t_*>0$. 
Without loss of generality, 
we can assume that $y$, $z$ and the origin are on the same straight line.
Then we apply the same argument as in the proof of Lemma~\ref{theorem:3.2} 
to define $F$-concave functions $\{\phi_\lambda\}_{\lambda>1}$.
Furthermore, by \eqref{eq:3.7} we find a sequence $\{\lambda_\ell\}_{\ell\in{\mathbb N}}$ such that 
\begin{align}
\begin{split}\label{b}
 & \left(e^{t\Delta_\Omega}\phi_{\lambda_{\ell}}\right)\left(\xi^{-1}\left(  (1-\tau)\lambda_\ell^{-1} y + \tau \lambda_\ell^{-1} z \right)\right)\\
 & <\min\left\{\left(e^{t\Delta_\Omega}\phi_{\lambda_{\ell}}\right)\left(\xi^{-1}(\lambda_\ell^{-1}y)\right),
 \left(e^{t\Delta_\Omega}\phi_{\lambda_{\ell}}\right)\left(\xi^{-1}(\lambda_\ell^{-1}z)\right)\right\}
\end{split}
\end{align}
for every large enough $\ell\in{\mathbb N}$ with $\lambda_\ell^{-1}y,\lambda_\ell^{-1}z\in B_0^{\mathbb{R}^n}(r)$,
where $(B_o^M(r),\xi)$ is a normal coordinate system at $o$.
Then, similarly to the proof of Lemma~\ref{theorem:3.2}, 
we obtain the desired conclusion. 
\end{proof}
\begin{lemma}
\label{theorem:3.4}
Assume $\lim_{r\to +0}F(r)>-\infty$. 
Then there exists $\infty$-concave function $\phi$ in~$\Omega$ such that $\|\phi\|_{L^\infty(\Omega)}<a$ and 
$e^{t_*\Delta_\Omega}\phi$ is not weakly $F$-concave in $\Omega$ for some $t_*>0$. 
In particular, 
neither $\infty$-concavity nor weak $\infty$-concavity is preserved by DHF in $\Omega$.
\end{lemma}
\begin{proof}
Let $a'\in(0,a)$ and set  
$\psi\coloneqq a'{\bf 1}_{B_o^{{\mathbb R}^n}(1)}$. 
Then $\psi$ is $\infty$-concave in ${\mathbb R}^n$. 
Since $(e^{\Delta_{{\mathbb R}^n}}\psi)(x)$ vanishes at infinity,
we find $x_*\in{\mathbb R}^n$ such that 
$$
(e^{\Delta_{{\mathbb R}^n}}\psi)(0)>(e^{\Delta_{{\mathbb R}^n}}\psi)(x_*).
$$
This together with the strict monotonicity of $F$ yields
$$
F((e^{\Delta_{{\mathbb R}^n}}\psi)(0))>F((e^{\Delta_{{\mathbb R}^n}}\psi)(x_*)). 
$$
If $e^{\Delta_{{\mathbb R}^n}}\psi\in{\mathcal C}_{{\mathbb R}^n}[F]$, then the function $f$ in $[0,\infty)$ 
defined by 
$$
f(s):=F(e^{\Delta_{{\mathbb R}^n}}\psi(s x_*))\quad \mbox{for $s\in[0,\infty)$}
$$
is concave in $[0,\infty)$. Then we see that 
$$
-\infty=\lim_{s\to\infty}f(s)=\lim_{r\to +0}F(r)>-\infty,
$$
which is a contradiction.
Thus $e^{\Delta_{{\mathbb R}^n}}\psi\notin{\mathcal C}_{{\mathbb R}^n}[F]$
and there exist  $y$, $z\in{\mathbb R}^n$ and $\tau\in(0,1)$ such that 
\begin{equation}
\label{eq:3.8}
F\left((e^{\Delta_{\mathbb{R}^n}}\psi)((1-\tau)y+\tau z)\right)
<(1-\tau)F\left((e^{\Delta_{\mathbb{R}^n}}\psi)(y)\right)
+\tau F\left((e^{\Delta_{\mathbb{R}^n}}\psi)(z)\right).
\end{equation}
Then we apply the same argument as in the proof of Lemma~\ref{theorem:3.2} 
to define $\infty$-concave functions $\{\phi_\lambda\}_{\lambda>1}$ in $\Omega$ by using $\psi$ instead of $\psi_m$.
Furthermore, by \eqref{eq:3.8} we have the same inequality as \eqref{eq:3.6}
and  obtain the desired conclusion.
\end{proof}
Combining these three lemmas, we have the following proposition.
\begin{proposition}
\label{theorem:3.1}
If either $\overline{F}$-concavity or weak $\overline{F}$-concavity is  preserved by DHF in $\Omega$,
then  $\overline{F}$-concavity is  preserved by DHF in $\mathbb{R}^n$.
\end{proposition}

Next, we show that the preservation of $\overline{F}$-concavity implies that of log-concavity.
\begin{proposition}
\label{theorem:3.5} 
Assume that some $F$-concavity \textup{(}resp.\,weak $F$-concavity\textup{)} is preserved by DHF in $\Omega$.
Then log-concavity \textup{(}resp.\,weak log-concavity\textup{)}is preserved by DHF in~$\Omega$.
\end{proposition}
\begin{proof}
Assume that some $F$-concavity (resp.\,weak $F$-concavity) is preserved by DHF in~$\Omega$.
Then 
$F|_{[0,a')}$-concavity (resp.\,weak $F|_{[0,a')}$-concavity)
is also preserved by DHF in~$\Omega$ for any $a'\in(0,a)$
hence we can assume  $a\in (0,\infty)$.
By Proposition~\ref{theorem:3.2} we find that $F$-concavity is preserved by DHF in~$\mathbb{R}^n$.
We observe  from Proposition~\ref{theorem:2.2}~\eqref{enumi:1} and Lemma~\ref{theorem:2.3}
that $F$-concavity is weaker than $H_{a}$-concavity and stronger than log-concavity in ${\mathcal A}_\Omega(a)$.

Let $f$ be a bounded log-concave (resp.\,weakly log-concave) function in $\Omega$.
By Lemma~\ref{theorem:2.5} we find a sequence $\{f_{b}\}_{b>0}$ such that 
$f_b$ is $H_b$-concave \textup{(}resp.\,weakly $H_{b}$-concave\textup{)}  in $\Omega$ and 
$\lim_{b\to\infty}f_b=f$ uniformly on $\Omega$.
For $b\in (a,\infty)$, set $\varepsilon_b\coloneqq ab^{-1}\in (0,1)$.
By definition we have
\[
H_a( \varepsilon_b f_b )=H(a^{-1} \varepsilon_b f_b )=H(b^{-1} f_b)=H_b(f_b)
\quad in\ \Omega.
\]
Thus  $\varepsilon_{b} f_{b}$ is $H_{a}$-concave (resp.\,weakly $H_a$-concave) in $\Omega$ hence also $F$-concave (resp.\,weakly $F$-concave).
Furthermore, the preservation of $F$-concavity (resp.\,weak $F$-concavity)
 by DHF in $\Omega$ implies that 
$e^{t\Delta_\Omega}\left(\varepsilon_{b} f_{b}\right)=\varepsilon_{b} e^{t\Delta_\Omega}f_{b}$ is $F$-concave (resp.\,weakly $F$-concave)
in $\Omega$ for every $t>0$.
We observe from the property $\mathcal{C}_{\mathbb{R}^n}[F] \subset  \mathcal{C}_{\mathbb{R}^n}[\Phi_0]$ with Lemma~\ref{theorem:2.3} 
that $\varepsilon_{b} e^{t\Delta_\Omega}f_{b}$ is log-concave (resp.\,weak log-concavity) in $\Omega$ for every $t>0$.
This implies that 
\[
\log\left(\left(e^{t\Delta_\Omega}f_{b}\right)(c_\tau )\right)
\geq
(1-\tau)\log\left(\left(e^{t\Delta_\Omega}f_{b}\right)(c_0 )\right)
+\tau \log\left(\left(e^{t\Delta_\Omega}f_{b}\right)(c_1)\right)
\]
for any (resp.\,some) minimal geodesic $c\colon [0,1]\to \Omega$ joining any two points in $\Omega$ and $\tau \in (0,1)$.
Letting $b\to\infty$, we have  
\[
\log\left(\left(e^{t\Delta_\Omega}f\right)(c_\tau)\right)
\geq
(1-\tau)\log\left(\left(e^{t\Delta_\Omega}f\right)(c_0)\right)
+\tau\log\left(\left(e^{t\Delta_\Omega}f\right)(c_1)\right),
\]
that is,  log-concavity (resp.\,weak log-concavity) is preserved by DHF in $\Omega$.
Thus the proposition follows.
\end{proof}
%
%
\section{Proof of Theorem~\ref{theorem:1.4}}\label{section:4}
In this section we discuss the non-preservation of  log-concavity and weak log-concavity by DHF in~$\Omega$,  
and complete the proof of Theorem~\ref{theorem:1.4}. 
We employ the arguments of~\cite{CW3}*{Lemma~3.1} with parabolic regularity theorems to prepare the following lemma.
\begin{lemma}
\label{theorem:4.1}
Fix $o\in \Omega$.
Let $\delta>0$ be such that $B_o^M(\delta)$ is strongly convex and  $B_o^M(\delta)\subset \Omega$.
Assume that there exist a bounded, smooth function $\psi$ on $B_o^M(\delta)$
and $v\in T_oM$ with the following three properties:
\begin{enumerate}
\renewcommand{\theenumi}{C\arabic{enumi}}
\renewcommand{\labelenumi}{{\rm(\theenumi)}}
\item\label{enumi:C1}
 $\psi$ is concave in $B_o^M(\delta)$;
 \item\label{enumi:C2}
 $\Hess_M \psi(v,v)=0$;
 \item\label{enumi:C3}
  $\Hess_M (\Delta_M \psi+ g(\nabla_M \psi,\nabla_M \psi))(v,v)>0$.
\end{enumerate}
Then $\phi\coloneqq e^{\psi}{\bf 1}_{B_o^M(\delta)}$ is bounded and log-concave in $\Omega$.
Furthermore, $e^{t\Delta_\Omega}\phi$ is not weakly log-concave in $\Omega$ for every small enough $t>0$.
\end{lemma}
\begin{proof}
It follows from \eqref{enumi:C1} that $\phi$ is bounded and log-concave in $\Omega$.
Notice that 
\[
\Delta_M \psi+ g(\nabla_M \psi,\nabla_M \psi)=\phi^{-1}\Delta_M \phi\quad \text{in }B_o^M(\delta).
\]
Applying parabolic regularity theorems (see \cite{LSU}*{Chapters~III and IV}, in particular, \cite{LSU}*{Chapter~IV, Theorem~10.1}),
we see that $e^{t\Delta_\Omega}\phi$ is $C^{4,\sigma;,2,\sigma/2}$-smooth in $B_\delta(o)\times[0,\infty)$ for some $\sigma\in(0,1)$.
Then it follows from \eqref{enumi:C3} that 
\begin{align*}
\frac{d}{d t} \Hess_M \left(\log e^{t \Delta_{\Omega}}\phi\right)(v,v)\bigg|_{t=0}
&=
\Hess_M \left(\frac{\partial}{\partial t} \log e^{t \Delta_{\Omega}} \phi \right)(v,v)\bigg|_{t=0}\\
&=\Hess_M \left(  \left( e^{t \Delta_{\Omega}}\phi\right)^{-1}  \frac{\partial}{\partial t} e^{t \Delta_{\Omega}}\phi \right)(v,v)\bigg|_{t=0}\\
&=\Hess_M \left( \left( e^{t \Delta_{\Omega}}\phi\right)^{-1} \Delta_M e^{t \Delta_{\Omega}}\phi \right)(v,v)\bigg|_{t=0}\\
& =\Hess_M\left(\phi^{-1} \Delta_M \phi\right)(v,v)>0.
\end{align*}
This together with  \eqref{enumi:C2} implies that
\[
\Hess_M \left( \log e^{t \Delta_{\Omega}}\phi \right)(v,v)>0
\]
for every small enough $t>0$ hence  $e^{t\Delta_\Omega}\phi$ is not weakly log-concave in $\Omega$ for every small enough $t>0$.
This completes the proof.
\end{proof}

Now we are ready to complete the proof of Theorem~\ref{theorem:1.4}.
\begin{proof}[Proof of Theorem~{\rm\ref{theorem:1.4}}]
Assume that the sectional curvature does not vanish at  $o\in \Omega$.
Fix $r>0$ such that $B_o^M(r)$ is strongly convex and $B_o^M(r)\subset \Omega$.
Let $\xi$ be a normal coordinate map on $B_o^M(r)$ determined by the orthonormal basis $e_1,\ldots, e_n$ for $T_oM$
such that the sectional curvature of some tangent plane containing  $e_1$ is not zero.
By the symmetry of $[R_{1i1j}(o)]_{2\leq i,j \leq n}$
we can assume that $e_2,\ldots, e_n$ are eigenvectors of $[R_{1i1j}(o)]_{2\leq i,j \leq n}$ 
hence
\begin{equation*}
R_{1i1j}(o)=-\kappa_i\delta_{ij}\quad\text{for }i,j=1,\ldots,n,
\end{equation*}
where $\kappa_1=0$ and $\kappa_i$ is the sectional curvature of the tangent plane spanned by $e_1$ and $e_i$ for $i=2,\ldots,n$.
Set 
\[
I_+:=\{i=2,\ldots, n\mid \kappa_i>0\}.
\]
By the assumption we find $I_+\neq \emptyset$.
Let $C>0$ satisfy 
\begin{align*}
Cg(v,v)+\frac{1}2 \sum_{k\in I_+}  \Hess_M \Gamma_{11}^k(v,v)&\geq 0 \quad \text{for  }v\in T_oM, \\
(n-1)\cdot \biggr( \kappa_i \mathbf{1}_{I_+}(i)+ n^{\frac{3}{2}} \max_{1\leq  k,\ell  \leq n} |R_{1\ell k i}(o)|    \biggr)^2 &\leq C \quad \text{for }i=2,\ldots, n.
\end{align*}

For  $\lambda>0$ to be specified later, define a bounded, smooth function $\psi$ on $B_o^M(r)$ by
\begin{align*}
\psi(p)\coloneqq &\,\, 2\lambda \sum_{i\in I_+}^n  \xi^i(p) -\lambda^2 \sum_{i=2}^n \xi^i(p)^2+\frac{2}{3}\lambda\xi^1(p)^2\sum_{i\in I_+}^n \kappa_i \xi^i(p)\\
& -\xi^1(p)^2\biggr[C(1 +\lambda)|\xi(p)|^2+\lambda^2 \sum_{i\in I_+}^n \kappa_i\xi^i(p)^2\biggr]
\quad
\text{for }p\in B_o^M(r).
\end{align*}
For $i=2,\ldots, n$, we calculate that 
\begin{align*}
 & \frac{\partial \psi}{\partial \xi^i}(o)=2\lambda \mathbf{1}_{I_+}(i), \quad
\frac{\partial^2 \psi}{\partial (\xi^i )^2}(o)=-2\lambda^2,\quad
\frac{\partial^3 \psi}{\partial (\xi^1)^2 \partial \xi^i}(o)=\frac{4}{3}\lambda \kappa_i\mathbf{1}_{I_+}(i),\\
 & \frac{\partial^4 \psi}{\partial (\xi^1)^2 \partial (\xi^i)^2}(o)=-4\left[C(1 +\lambda)+\lambda^2 \kappa_i\mathbf{1}_{I_+}(i)\right],\quad
\frac{\partial^4 \psi}{\partial (\xi^1)^4}(o)=-24 C(1 +\lambda),
\end{align*}
and all the other partial derivatives of $\psi$ at $o$ are zero.
In particular, we have
\[
\left[ \Hess_M \psi(e_i, e_j) \right]_{1\leq i,j \leq n} =
\left[ \frac{\partial^2 \psi}{\partial \xi^i \partial \xi^j}(o) \right]_{1\leq i,j \leq n}
=\diag[0, -2\lambda^2,\ldots,-2\lambda^2],
\]
which implies that $\psi$ satisfies \eqref{enumi:C2} for $v=e_1$.

We show that there exists $\delta \in (0,r)$ such that $\psi$ satisfies~\eqref{enumi:C1}, that is, 
\begin{align*}
H(\tau)=[h_{ij}(\tau)]_{1\leq i,j \leq n}
&\coloneqq 
\left[ \Hess_M \psi \left(\frac{\partial }{\partial \xi^i}, \frac{\partial }{\partial \xi^j}\right)(\exp_o(\tau v)) \right]_{1\leq i,j \leq n}\\
&=\left[ \frac{\partial^2 \psi}{\partial \xi^i\partial \xi^j}(\exp_o(\tau v)) 
-\sum_{k=1}^n\Gamma_{ij}^k(\exp_o(\tau v))  \frac{\partial \psi}{\partial \xi^k}(\exp_o(\tau v)) \right]_{1\leq i,j \leq n}
\end{align*}
is nonpositive definite for any unit tangent vector $v$ at $o$ and $\tau \in [0,\delta)$.
Thanks to 
\[
H(0)=\diag[0, -2\lambda^2,\ldots,-2\lambda^2], 
\]
it is enough to show 
\begin{align}
\label{eq:4.1}
  (-1)^{n-1} \frac{d}{d\tau} \det H(\tau)\bigg|_{\tau=0}&=0,\\
\label{eq:4.2}
  (-1)^{n-1} \frac{d^2}{d\tau^2} \det H(\tau)\bigg|_{\tau=0}&<0.
 \end{align}
We denote by $\widetilde{H}(\tau)=[\widetilde{h}_{ij}(\tau)]_{1\leq i,j \leq n}$ the adjugate matrix of $H(\tau)$.
Then we have
\[
\widetilde{H}(0)=\diag[(-2\lambda^2)^{n-1}, 0,\ldots,0]
\]
and deduce from Jacobi's formula that
\begin{align*}
\frac{d}{d\tau} \det H(\tau)\bigg|_{\tau=0}
&=\tr  \left(\widetilde{H}(\tau) H'(\tau)\right)\bigg|_{\tau=0}
=(-2\lambda^2)^{n-1} h_{11}'(0),\\
\frac{d^2}{d\tau^2} \det H(\tau)\bigg|_{\tau=0}&=\tr \left( \widetilde{H}(0)H''(0) \right)+ \tr \left(\widetilde{H}'(0) H'(0) \right)\\
& =(-2\lambda^2)^{n-1} h_{11}''(0)+ \sum_{i,j=1}^n \widetilde{h}_{ij}'(0){h}'_{ji}(0).
\end{align*}
By the Taylor expressions of the Christoffel symbols  (see~\eqref{eq:2.1}) we compute
\[
 \Gamma_{11}^k(o)=0, \quad \frac{\partial \Gamma_{11}^k}{\partial \xi^\ell}(o)
=\frac13\sum_{\ell=1}^n \left(R_{1 \ell k 1 }(o) +R_{1 \ell k 1 }(o) \right)  
=\frac23 \kappa_k \delta_{k\ell} \quad \text{for }k,\ell=1,\ldots,n
\]
hence 
\[
h_{11}'(0)
=\sum_{\ell=1}^n\frac{\partial}{\partial \xi^\ell}\left( \frac{\partial^2 \psi}{\partial (\xi^1)^2} -\sum_{k=1}^n\Gamma_{11}^k  \frac{\partial \psi}{\partial \xi^k}\right)\bigg|_o v^\ell
=\frac43\lambda\sum_{\ell\in I_+} \kappa_\ell v^\ell-\frac43\lambda\sum_{k\in I_+}  \kappa_k  v^k=0,
\]
which shows \eqref{eq:4.1}.
Let us next prove \eqref{eq:4.2}.
A direct computation with the definition of $C>0$ gives
\begin{align}
\begin{split}
\label{eq:4.3}
 h''_{11}(0)
=&\sum_{\ell, m=1}^n\frac{\partial^2}{\partial \xi^\ell\partial \xi^m}\left( \frac{\partial^2 \psi}{\partial (\xi^1)^2} -\sum_{k=1}^n\Gamma_{11}^k  \frac{\partial \psi}{\partial \xi^k}\right)\bigg|_o v^\ell v^m\\
=&
\sum_{\ell,m=1}^n \frac{\partial^4 \psi}{\partial (\xi^1)^2 \partial \xi^\ell \partial \xi^{m}}(o)v^\ell v^{m}\\
&-2\sum_{k,\ell,m=1}^n \frac{\partial \Gamma_{11}^k}{\partial \xi^\ell}(o) \frac{\partial^2 \psi}{\partial \xi^k \partial\xi^m}(o) v^\ell v^m
-\sum_{k=1}^n\Hess_M \Gamma_{11}^k (v,v)\frac{\partial \psi}{\partial \xi^k}(o)\\
=&-4\left[ 5C(1 +\lambda)(v^1)^2+C(1 +\lambda)g(v,v) +\lambda^2 \sum_{\ell\in I_+}^n  \kappa_\ell (v^\ell)^2\right]\\
&
+\frac{8}{3}\lambda^2\sum_{k=2}^n \kappa_k (v^k)^2
 -2\lambda \sum_{k\in I_+}\Hess_M \Gamma_{11}^k(v,v)\\
=&-4C\left[ 5(1+\lambda)(v^1)^2+g(v,v) \right]\\
&+\frac43\lambda^2 \sum_{i=2}^n \kappa_i \left(2-3\mathbf{1}_{I_+}(i)\right)(v^i)^2
-4\lambda \left[ C g(v,v) +\frac{1}2 \sum_{k\in I_+}  \Hess_M \Gamma_{11}^k(v,v)\right]\\
\leq &-4C\left[ 5(1 +\lambda)(v^1)^2+g(v,v) \right]\\
\leq &-4C,
\end{split}
\end{align}
where the inequality follows from the property $\kappa_i (2-3\mathbf{1}_{I_+}(i))\leq 0$ for $i=2,\ldots,n$.
To compute $\tr (\widetilde{H}'(0)H'(0))$,
we observe from \eqref{eq:4.1} that
\[
H(0)\widetilde{H}'(0)+H'(0)\widetilde{H}(0)=
\frac{d}{d\tau}\left(H(\tau)\widetilde{H}(\tau)\right)\bigg|_{\tau=0}=
\frac{d}{d\tau}(\det H(\tau) E_n)\bigg|_{\tau=0}=O_n.
\]
By the properties $H(0)=\diag[0,  -2\lambda^2,\ldots, -2\lambda^2]$ and $\widetilde{H}(0)=\diag[(-2\lambda^2)^{n-1}, 0,\ldots,0]$
this implies that $\widetilde{h}'_{ij}(0)\neq 0$ happens only for either $i=1$ or $j=1$ 
hence
\begin{align*}
\sum_{i,j=1}^n \widetilde{h}_{ij}'(0){h}'_{ji}(0)
=\sum_{j=2}^n \widetilde{h}_{1j}'(0){h}'_{j1}(0)
+\sum_{i=2}^n \widetilde{h}_{i1}'(0){h}'_{1i}(0)
 =2\sum_{i=2}^n \widetilde{h}_{i1}'(0){h}'_{1i}(0),
\end{align*}
where we used $h'_{11}(0)=0$ in the first equality and the symmetry of $H(\tau)$ and  $\widetilde{H}(\tau)$ in the last equality.
Let $i=2,\ldots, n$. 
We observe from the cofactor expansion and  the property $H(0)=\diag[0,  -2\lambda^2,\ldots, -2\lambda^2]$ that 
\begin{equation*}
\widetilde{h}'_{i1}(0)
=(-1)^{1+i}\cdot (-2\lambda^2)^{n-2} \cdot (-1)^{i-2}\cdot{h}_{1i}'(0)
=(-1)^{n-1}\cdot (2\lambda^2)^{n-2} \cdot {h}_{1i}'(0).
\end{equation*}
We also find 
\begin{align*}
h_{1i}'(0)
&=\sum_{\ell=1}^n\frac{\partial}{\partial \xi^\ell}\left( \frac{\partial^2 \psi}{\partial \xi^1 \partial \xi^i} -\sum_{k=1}^n\Gamma_{1i}^k  \frac{\partial \psi}{\partial \xi^k}\right)\bigg|_o v^\ell\\
&=\frac43\lambda \kappa_i \mathbf{1}_{I_+}(i) v^1
-
\frac23\lambda \sum_{k,\ell=1}^n (R_{1\ell ki}(o)+R_{i\ell k 1} (o))\mathbf{1}_{I_+}(k)v^{\ell},
\end{align*}
which yields 
\begin{align*}
\left| h_{1i}'(0) \right| \leq  \frac43\lambda\left( \kappa_i \mathbf{1}_{I_+}(i)+ n^{\frac{3}{2}}\max_{1\leq  k,\ell  \leq n} |R_{1\ell k i}(o)|    \right)  |v| \leq \frac43\lambda \sqrt{\frac{C}{n-1}}.
\end{align*}
Combining this with  \eqref{eq:4.3}, we obtain
\begin{align*}
(-1)^{n-1}\frac{d^2}{d\tau^2} \det H(\tau)\bigg|_{\tau=0}
& =(2\lambda^2)^{n-1} h_{11}''(0)+ (2\lambda^2)^{n-2} \cdot \lambda^{-2}\sum_{i=2}^n h_{1i}'(0)^2\\
&\leq (2\lambda^2)^{n-1} \cdot\left(
-4C+\frac{16}{9} C \right)<0,
\end{align*}
which leads to \eqref{eq:4.2}.
Thus $\psi$ satisfies \eqref{enumi:C1} for every small enough $\delta>0$.

Let us verify \eqref{enumi:C3} for $v=e_1$, and complete the proof of Theorem~\ref{theorem:1.4}.
We compute 
\begin{align*}
&\mathrm{Hess}_M (\Delta_M \psi+ g(\nabla_M \psi,\nabla_M \psi)) (e_1, e_1)\\
& =\frac{\partial^2 }{\partial (\xi^1)^2}
\left[\, \sum_{i,j=1}^n g^{ij} 
\biggr(\frac{\partial^2 \psi}{\partial\xi^i\partial \xi^j} + \frac{\partial \psi}{\partial \xi^i} \frac{\partial \psi}{\partial \xi^j}-\sum_{k=1}^n \Gamma_{ij}^k \frac{\partial \psi}{\partial \xi^k}  \biggr) \right]\bigg|_o\\
& =\sum_{i,j=1}^n \frac{\partial^2 g^{ij}}{\partial (\xi^1)^2}(o)
\left(
\frac{\partial^2 \psi}{\partial\xi^i\partial \xi^j}(o)
+
\frac{\partial \psi}{\partial \xi^i}(o) \frac{\partial \psi}{\partial \xi^j}(o)
\right)\\
& \quad
+\sum_{i=1}^n \frac{\partial^2 }{\partial (\xi^1)^2}
 \biggr[
\frac{\partial^2 \psi}{\partial(\xi^i)^2}
+
\left(\frac{\partial \psi}{\partial \xi^i}\right)^2
-\sum_{k=1}^n \Gamma_{ii}^k \frac{\partial \psi}{\partial \xi^k}
\biggr]\bigg|_o\\
& =\frac{2}{3}\sum_{i=1}^n\kappa_i (-2\lambda^2+4\lambda^2 \mathbf{1}_{I_+}(i))
\\
& \quad
-4\biggr[C(n+5)(1 +\lambda)+\lambda^2 \sum_{i\in I_+}\kappa_i\biggr]
+\frac{16}{3}\lambda^2 \sum_{i\in I_+} \kappa_i
-2\lambda \sum_{k\in I_+}\sum_{i=1}^n \Hess_M\Gamma_{ii}^k (e_1,e_1)\\
& =\frac{4}{3}\lambda^2\sum_{i=1}^n\kappa_i (-1+3 \mathbf{1}_{I_+}(i))
-4\lambda\biggr[C(n+5)+\frac12 \sum_{k\in I_+}\sum_{i=1}^n \Hess_M\Gamma_{ii}^k (e_1,e_1)\biggr]
-4C(n+5),
\end{align*}
where we used the properties (see~\eqref{eq:2.1})
\[
g^{ij}(o)=\delta^{ij}, \quad
\frac{\partial g^{ij}}{\partial \xi^1}(o)=0,
\quad
\frac{\partial g^{ij}}{\partial (\xi^1)^2}(o)=-\frac23 R_{i 1 j 1}=\frac23\kappa_{i}\delta_{ij},
\quad \text{for }i,j=1,\ldots,n.
\]
Since  $\kappa_i(-1+3\mathbf{1}_+(i))\geq 0$ with equality if and only if $\kappa_i=0$,
we observe from the assumption $I_+\neq \emptyset$ that
$\mathrm{Hess}_M (\Delta_M \psi+ g(\nabla_M \psi,\nabla_M \psi)) (e_1, e_1)$ is a quadratic function of~$\lambda $
and there exists $\lambda>0$ such that 
\[
\mathrm{Hess}_M (\Delta_M \psi+ g(\nabla_M \psi,\nabla_M \psi)) (e_1, e_1)>0,
\]
that is, $\psi$ satisfies \eqref{enumi:C3} for $v=e_1$. 
Thus $\psi$ satisfies \eqref{enumi:C1}--\eqref{enumi:C3} for $v=e_1$, 
consequently, log-concavity is not preserved by DHF in $\Omega$ by Lemma~\ref{theorem:4.1}.
Then the theorem follows from Proposition~\ref{theorem:3.5}.
\end{proof}

Similarly, we obtain the following result on a weak $\overline{F}$-concavity version of Theorem~\ref{theorem:1.4}.
\begin{theorem}
\label{theorem:4.2}
If some weak $\overline{F}$-concavity is preserved by DHF in $\Omega$,
then the sectional curvature is identically zero on $\Omega$. 
\end{theorem}
\noindent
{\bf Acknowledgements.} 
K.I. and A.T. were supported in part by JSPS KAKENHI Grant Number 19H05599.
A.T. was supported in part by JSPS KAKENHI Grant Number 19K03494.
\medskip

\noindent
{\bf  Conflict of Interest.}
The authors state no conflict of interest. 

\end{document}